\documentclass[12pt]{amsart}
\usepackage[utf8]{inputenc}

\usepackage{amscd}

\usepackage[all]{xy}
\usepackage{color}

\numberwithin{equation}{section}

\newtheorem{theorem}[equation]{Theorem}
\newtheorem*{theorem*}{Theorem}

\newtheorem{proposition}[equation]{Proposition}

\newtheorem{lemma}[equation]{Lemma}

\newtheorem{corollary}[equation]{Corollary}

\theoremstyle{remark}
\newtheorem{remark}[equation]{Remark}

\theoremstyle{definition}

\def\XXint#1#2#3{{\setbox0=\hbox{$#1{#2#3}{\int}$}
	\vcenter{\hbox{$#2#3$}}\kern-.5\wd0}}

\newcommand{\N}{\mathbb N}

\newcommand{\R}{\mathbb R}

\newcommand{\norm}[1]{\left\Vert#1\right\Vert}

\def\eps{\epsilon}
\def\half{{1 \over 2}}

\begin{document}

\title[Isometric embeddings of snowflakes] 
{Isometric embeddings of snowflakes into finite-dimensional Banach spaces}

\author{Enrico Le Donne}

\author{Tapio Rajala}
\author{Erik Walsberg}

\address[Le Donne, Rajala]{University of Jyvaskyla\\
         Department of Mathematics and Statistics \\
         P.O. Box 35 (MaD) \\
         FI-40014 University of Jyvaskyla \\
         Finland}
\email{enrico.e.ledonne@jyu.fi}
\email{tapio.m.rajala@jyu.fi}

\address[Walsberg]{Department of Mathematics\\ University of Illinois at Urbana-Champaign \\1409 West Green Street\\ Urbana\\ Il\\ 61802\\ USA}
\email{erikw@illinois.edu}

\thanks{E.L.D. and T.R. acknowledge the support of the Academy of Finland, projects no. 288501 and 274372. 
E.W acknowledges the support of the European Research Council under the European Union's Seventh Framework Programme (FP7/2007-2013) / ERC Grant agreement no.\ 291111/ MODAG}


\subjclass[2010]{
30L05; 
46B85; 
54C25; 
54E40; 
28A80. 
}

\date{September 12, 2016}

\begin{abstract}

%

We   consider a  general notion of snowflake  of a metric space by composing the distance  by a nontrivial concave function. 
 We prove that 
a snowflake of a metric space $X$   isometrically embeds into some finite-dimensional normed space if and only if $X$ is finite.
In the case of 
  power functions
we give a uniform bound on the cardinality of $X$ depending only on
the power exponent and the dimension of the vector space.

\end{abstract}

\maketitle


\section{Introduction}
Isometric embeddings of metric spaces into infinite-dimensional Banach spaces have a long tradition. 
Classical  results are due to
Fr\'echet, Urysohn,  Kuratowski,  Banach,
 \cite{
  Frechet1910,  
 Urysohn1927,    
Kuratowski1935,  
Banach1955}. 
For an introduction to the subject we refer to Heinonen's survey \cite{MR2014506}.
The case of  embeddings into
finite-dimensional Banach spaces is much harder, even when one considers bi-Lipschitz embeddings in place of isometric embeddings.
It is a wide open problem to give an intrinsic characterizations of those metric spaces which admit bi-Lipschitz embeddings into some Euclidean space. See for example
\cite{MR1744622,    
MR1387955,    
Lang-Plaut,    
MR2915474,    
MR3253787}.  

The situation is quite different for quasisymmetric maps (see \cite[Chapter~10-12]{Heinonenbook} for an introduction to the theory of quasisymmetric embeddings).
A metric space quasisymmetrically embeds into some Euclidean spaces if and only if it is doubling (see \cite[Theorem~12.1]{Heinonenbook}).
More specifically, Assouad proved the following result (see \cite{Assouad83}, and also \cite{MR2990137, MR3108866}):
if $(X,d)$ is a doubling metric space and $\alpha\in (0,1)$ then the metric space $(X,d^\alpha)$ admits a bi-Lipschitz embedding into some Euclidean space.
If $d_E$ is the Euclidean distance and $\alpha=\log 2 / \log 3$ then the metric space
$([0,1], d_E^\alpha )$ is bi-Lipschitz equivalent to the von Koch snowflake curve, so
 $(X,d^\alpha)$  is said to be the {\em $\alpha$-snowflake} of  $(X,d)$.

The Assouad Embedding Theorem  is sharp in that there 
  are examples (none of which are trivial) of doubling spaces which do not admit bi-Lipschiz embeddings into any Euclidean space,
  even though each of their $\alpha$-snowflakes do.
See  
\cite{
Semmes, 
Semmes2,   
Laakso,   
Cheeger-Kleiner2}. 
We also  stress that it has been known that snowflakes of doubling   spaces
in general do not isometrically  embed  in any Euclidean space.
Indeed,  the space
  $([0,1], d_E^{1/2})$ does not, see
  \cite[Remark~3.16(b)]{MR2014506}.
  
  The main aim of this paper
is to show that if some $\alpha$-snowflake of a metric space
isometrically embedds into a finite dimensional Banach space then the
metric space in question is finite.
  Our main result is the following.
  
\begin{theorem}\label{thm:main}
For any $n \in \N$ and $\alpha \in (0, 1)$ there is an $N \in \N$ such that if a metric space $(X,d)$ has cardinality at least $N$ then    $(X,d^\alpha)$ does not admit an isometric embedding into any $n$-dimensional normed linear space.
\end{theorem}

The techniques that we use in Section~\ref{sec:thm:main1} for the proof of Theorem~\ref{thm:main} can also be used to study more general notions of snowflakes.
For this purpose, we introduce general   snowflaking functions.
We say that a function $h: \mathbb{R}_{\geq} \to \mathbb{R}_{\geq}$ is a  {\em snowflaking function} if the following hold:
\begin{enumerate}
\item[(S1)] $h(0) = 0$.
\item[(S2)] $h$ is concave.
\item[(S3)] $ \frac{h(t)}{t} \to \infty$, as ${t \to 0}$.
\item[(S4)] $ \frac{h(t)}{t} \to 0$, as ${t \to \infty} $.
\end{enumerate}
Let $h$ be a snowflaking function.
Then function $h$ is  weakly increasing and, if $d$ is a metric on a set $X$ then $h \circ d$ is also a metric on $X$.
Given a snowflaking function $h$ and a metric space $(X,d)$ 
we say that the metric space $(X, h \circ d)$ is the  {\em $h$-snowflake} of $(X,d)$.
If $h(t) = t^\alpha$ for some $\alpha \in (0,1)$ then $(X, h \circ d)$ is the $\alpha$-snowflake of $(X,d)$.
In Section~\ref{sec:thm:main} we prove the following.
\begin{theorem}\label{thm:main1}
Let $h$ be a snowflaking function and $(X,d)$ a metric space.
If the $h$-snowflake of $(X,d)$  admits an isometric embedding into some finite-dimensional Banach space then $X$ is finite.
\end{theorem}
\begin{remark}
Note that for general snowflaking functions there may not be any bound on the number of points one can embed, see Remark \ref{rmk:nobound}.
If one removes either of the requirements (S3) or (S4), then we say that $(X, h \circ d)$ is a degenerate snowflake (at zero or at infinity, respectively). Indeed, in such cases the conclusion of Theorem~\ref{thm:main1} does not hold, in general, see Proposition \ref{prop:necessity}.
\end{remark}
 
We conclude the introduction with few other simple observations about embeddings into Euclidean spaces.
Every $\alpha$-snowflake of $\mathbb{R}^n$, $\alpha\in (0,1)$, isometrically embeds
 into the Hilbert space $\ell^2$ of square summable sequences, see \cite[Remark~3.16(d)]{MR2014506}.
 For any $n\in \N$, there is a metric space of cardinality $n$ such that for any $\alpha \in (0,1)$ its $\alpha$-snowflake 
can be isometrically embedded into $\mathbb{R}^{n-1}$ (just take the vertices of the standard simplex).
There is a 4-point metric space which has an $\alpha$-snowflake which cannot be isometrically embedded into $\mathbb{R}^4$, 
and so cannot be isometrically embedded into any Euclidean space, (just take the vertices of the (3,1) complete bipartite graph).
Every finite metric space has an $\alpha$-snowflake which admits an isometric embedding into some Euclidean space, see Proposition~\ref{prop:existence} below.

\section{Corollaries}

%
%
%

In this section we record a few small results and corollaries.
First of all, note that as a corollary of Theorem \ref{thm:main1},  by Ascoli-Arzel\`a Theorem we immediately obtain:
\begin{corollary}\label{lipschitz}
Suppose that $(X,d)$ is infinite.
For any snowflaking function $h$ and $n \in \mathbb{N}$ there is a $\delta > 0$ such that $(X, h \circ d)$ does not admit a $(1 + \delta)$-bilipschitz embedding into any normed linear space of dimension $n$.
\end{corollary}

Theorem \ref{thm:main} shows that there is a bound (depending on $\alpha$ and $n$) on the cardinality of a metric space whose $\alpha$-snowflake admits an isometric embedding into $\R^n$.
The next easy result says that any finite metric space can be isometrically embedded in some Euclidean space after some $\alpha$-snowflaking.
\begin{proposition}\label{prop:existence}
If $(X,d)$ is a finite metric space of cardinality $n$, then some $\alpha$-snowflake of $(X,d)$ admits an isometric embedding into the Euclidean space 
 $\R^n$.
\end{proposition}

Combining Theorem~\ref{thm:main} and Proposition~\ref{prop:existence} we have:\begin{corollary} 
Given a metric space $(X,d)$,
there is an $\alpha$-snowflake of $(X,d)$ that isometrically embedds into some finite-dimensional normed space if and only if $X$ is finite.
\end{corollary}

\begin{proof}[Proof of Proposition \ref{prop:existence}]
We show that $(X,d^\alpha)$ admits an isometric embedding into $\R^n$ when $\alpha$ is sufficiently close to $1$.
In fact we show that there is an $\eps > 0$ such that if $d'$ is any metric on $X$ satisfying $1 - \epsilon < d'(x,y) < 1 +\epsilon$ for all distinct $x,y \in X$ then $(X,d')$ admits an isometric embedding into $\R^n$.
The proposition follows as $d(x,y)^\alpha \to 1$ as $\alpha \to 0$ for any distinct $x,y \in X$.

Consider the points $p_j:= \tfrac{1}{\sqrt{2}} e_j$ where $e_1, \ldots, e_n$ is the canonical basis of $\R^n$.
Thus $d(q_i,q_j) = 1$ for all distinct $1 \leq i,j \leq n$.
The proposition is thus a direct consequence of the following claim:

{\bf Claim A.} There is an $\eps>0$ such that if $\rho_{ij}\in \R$, with $1\leq i<j\leq n$ are such that
$|\rho_{ij} - 1 | < \eps$, then there exist $q_1, \ldots, q_n\in \R^n$ such that
$\norm{q_i-q_j} = \rho_{ij}$ for all $1 \leq i < j \leq n$.

To show the claim, consider
the vector subspace $U$ of $\R^{n\times n}
$
defined by the upper-triangular matrices, i.e., 
the elements of $ U$ are elements of the form $  (a_{ij})_{ij}$ with 
$1\leq i<j\leq n$ and $a_{ij}\in \R$.
Let
$F:\R^{n\times n} \to U$ be given by
$$F(q_1, \ldots, q_n) := (\norm{q_i-q_j}^2)_{ij} \quad \text{ for }\,q_1, \ldots, q_n\in \R^n.$$
Denote by
${\bf p}$ the vector of vectors 
${\bf p}:= (\tfrac{1}{\sqrt{2}} e_1, \ldots, \tfrac{1}{\sqrt{2}} e_n)$.
Notice that
$F({\bf p})=(1)_{ij}$.

Claim A holds if and only if $F(\bf p)$ lies in the interior of the image of $F$.
Thus it suffices to show that the differential $(dF)_{\bf p}$ has maximal rank and apply the inverse function theorem.

Denote by $q_l^k$ the $k$-th component of an $n$-vector $q_l$.
Notice that $e_l^k= \delta_l^k$, where $\delta_l^k$ is the 
Kronecker symbol.
Set $E_{lk} = (\delta_i^l \delta_j^k)_{ij}$ Notice that
$\{E_{lk}, 1\leq i<j\leq n\}$ span the set $U$.
Let us show that any $E_{lk}$ is in the image of $(dF)_{\bf p}$.
Fix $l$ and $k$ so that 
$1\leq l<k\leq n$.
By direct calculation,
\begin{eqnarray*}
\dfrac{\partial F( {\bf q})}{\partial q_l^k}  = 
\left( 2 \langle q_i -q_j , 
\delta_i^l e_k - \delta_j^l e_k\rangle \right)_{ij}
 = 
\left( 2 ( q_i^k -q_j^k )
(\delta_i^l - \delta_j^l ) \right)_{ij}.
\end{eqnarray*}
Evaluating at ${\bf p}$, we get
$$
\left.\dfrac{\partial F( {\bf q})}{\partial q_l^k} \right|_{{\bf q}={\bf p}} = 
\left( 2 \frac{1}{\sqrt{2}}( \delta_i^k -\delta_j^k )
(\delta_i^l - \delta_j^l ) \right)_{ij}.
$$
Since $l\neq k$, we have 
$ \delta_i^k\delta_i^l\equiv 0$ and 
$\delta_j^k\delta_j^l \equiv 0$.
Since $l< k$, and $i<j$, we have 
$ \delta_i^k\delta_j^l \equiv 0$.
Hence 
$$
\left.\dfrac{\partial F( {\bf q})}{\partial q_l^k} \right|_{{\bf q}={\bf p}} = 
\left(  \sqrt{2} ( -\delta_j^k )
\delta_i^l \right)_{ij} = - \sqrt{2}  E_{lk}.
$$
This concludes the proof of Claim A.
\end{proof}

\section{Proofs of the main results}

In the next subsection we first give two simple geometric lemmas which allow us to prove our results for non-euclidean normed vector spaces. The proofs of Theorem~\ref{thm:main} and Theorem~\ref{thm:main1} rely on the Ramsey-theoretic fact that given sufficiently many points in $\R^n$
there must be a triple that forms as large angle as we want. We recall this result at the end of the subsection in Lemma~\ref{lma:angles}.
In the following two subsections we then prove Theorem~\ref{thm:main} and Theorem~\ref{thm:main1}.
In the final subsection we prove that if $h$ is a degenerate snowflaking function then there exists an infinite metric space whose $h$-snowflake isometrically embeds into $2$-dimensional Euclidean space.

\subsection{Geometric lemmas}

Throughout this section $V$ is an $n$-dimensional normed linear space with norm $\|\cdot\|$. 
For the remainder of this section we fix   an inner product $\langle\cdot,\cdot\rangle$ on $V$ for which John ellipsoid property holds. Namely, we have
\begin{equation}\label{John:eq}
B \subset B_{V} \subset \sqrt{n}B,
\end{equation}
where $B_{V}$ is the $\|\cdot\|$-unit ball and $B$ is the the unit ball in the $l_2$ metric associated to the inner product $\langle\cdot,\cdot\rangle$. For $u,v \in V$ we denote the length given by the inner product by $uv := \sqrt{\langle u-v,u-v\rangle}$.

Given two vectors $u,v \in V$ let $\angle(u,v)$ be their angle with respect to the above inner product $\langle\cdot,\cdot\rangle$:
\[\angle(u,v) := \arccos \dfrac{ \langle u,v\rangle}{ \langle u,u\rangle  \langle v,v\rangle}.\]
Given three points $x,y,z
\in V$ we set
$\angle_y(x,z) := \angle (x-y, z-y)$.

Let $x,y,z \in V$ and let  $z' := x + \langle z-x,y-x \rangle(y-x)$ be the projection of $z$ onto the line through $x$ and $y$.

\begin{lemma}\label{lma:closetoEuclidean}
There exist constants $    \varepsilon , C \in (0, \infty)$ 
 depending only on $n$ such that 
 if $\angle_x(y,z) < \varepsilon$ then
 \[
  \big|\|x-z\| - \|x-z'\|\big| \le C\|x-z'\|\angle_x(y,z).
 \]
\end{lemma}
\begin{proof}
By the John Ellipsoid Theorem, \eqref{John:eq}, there is an angle $\theta$ and a height $h$ only depending on $n$ such that
for any point $p$ on the sphere with respect to $\norm{\cdot}$. 
The cone $C^- (p, \theta, \ell, 0)$ at $p$ with opening $\theta$ and  height $\ell$ in the direction of $0$ is inside the ball, whereas
the  cone $C^+ (p, \theta, \ell, p)$ at $p$ with opening $\theta$ and  height $\ell$ in the direction of $p$ is outside the ball.

The inequality that we need to prove is translation and dilation invariant. So we assume that $x=0$ and that $z'$ lies on the unit sphere.
Let $\beta := \angle_x(y,z)$. We shall require $\beta  < \varepsilon < \frac\pi2$ for a small enough $\eps$.
Let $t, s, t_1, t_2, t_3$ be as in the picture.
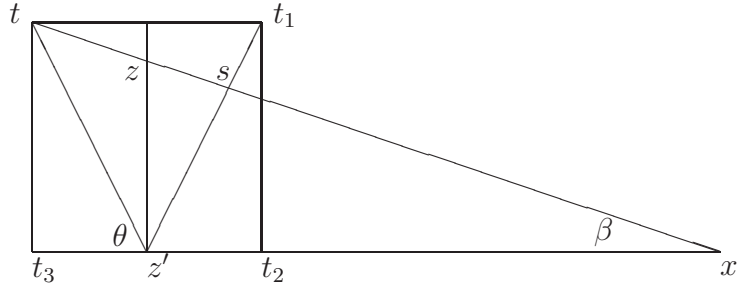
\begin{figure}[!h]
  \centering
  \setlength{\unitlength}{0.1\textwidth}
  \begin{picture}(5,1.9)
  
   \put(0,2){\line(3,-1){6}}   
       \put(0,0){\line(1,0){6}}  
        \put(0,2){\line(0,-2){2}}  
                        \put(0,2){\line(1,0){2}}  
                            \put(2,2){\line(0,-2){2}}  
                            \put(1,2){\line(0,-2){2}}  
                                \put(6,-0.2){$x$}
                                                                \put(4.9,0.1){$\beta$}
           \put(1,0){\line(1,2){1}} 
           \put(1,0){\line(-1,2){1}} 
    \put(2,-0.22){$t_2$}
        \put(1,-0.22){$z'$}
                \put(0.7,0.05){$\theta$}
                \put(0,-0.22){$t_3$}
                    \put(-0.2,2){$t$}
                                        \put(2.1,2){$t_1$}
    \put(0.8,1.5){$z$}
        \put(1.6,1.5){$s$}
          \end{picture} 
  \caption{The point $w$ is between $t$ and $s$.}
\end{figure}

If $\eps$ is small enough, then
$\sin \theta \cos \beta - 2 \cos \theta \sin \beta > \half\sin\theta.$
Setting $r:= tz'$ we have
$$t_2y = t_3y- t_3t_2 = r \dfrac{\sin\theta}{\sin \beta} \cos \beta - 2 r \cos \theta> \dfrac{r\sin\theta}{2\sin\beta}.$$
Also, if $\eps$ is small enough, then the point $t$ is in $C^+ (p, \theta, \ell, p)$ and so $s\in C^- (p, \theta, \ell, 0)$.
In particular, between $s$ and $t$ there is a point $w$ such that
$\|x-w\| = \|x-z'\|$.
 
We conclude using, in order, the definition of $w$, the fact that on a line any two norms are a multiple of each other, the triangle inequality, the properties of $s$ and $t$, and the previous bound:
\begin{eqnarray*}
1& \leq &\dfrac{\|x-z\| }{ \|x-z'\|} =
\dfrac{\|x-z\| }{ \|x-w\|} 
= 
\dfrac{xz }{ xw} 
\leq 
\dfrac{xw +wz }{ xw}
=
1+ \dfrac{wz }{ xw} 
\leq  1+ \dfrac{ts }{ sy} \\
&\leq & 1+ \dfrac{tt_1 + tz' }{ t_2y}
\leq  1+ \dfrac{ 2r \cos \theta + r }{  \frac{r}{2}\sin\theta / \sin \beta}
\leq  1+   C_\theta \sin \beta, \\
\end{eqnarray*}
where $C_\theta  =  ( 4  \cos \theta +2) /  \sin\theta    $.
\end{proof}

The following bound is another easy consequence of the John Ellipsoid theorem.
\begin{lemma}\label{lemma234124}
There exists a constant $K>0$ depending only on $n$ such that 
when $  \angle_x(y,z),\angle_y(x,z) \in(0, \frac{\pi}{4})$ we have
 \begin{equation}\label{eq:anglecomparison}
  K^{-1}\frac{\|z'-y\|}{\|z'-x\|} \le \frac{\angle_x(y,z)}{\angle_y(x,z)} \le K \frac{\|z'-y\|}{\|z'-x\|}.
 \end{equation}

\end{lemma}
\begin{proof}
 Set $\alpha =   {\angle_x(y,z)}$, and $ \beta:={\angle_y(x,z)}$. By \eqref{John:eq}
 \begin{eqnarray*}
 \frac{\alpha}{\beta} &\le& \frac{\tan\alpha}{\half\tan\beta} 
  =2 \frac{ zz'/xz' }{  zz'/yz'} 
  =2 \frac{  yz'} { xz' }
 \le 2\frac{\sqrt{n} \|z'-y\|}{\|z'-x\|}.
 \end{eqnarray*}
By symmetry, we get the other inequality.
 \end{proof}
 
Our arguments rely on the following Ramsey-theoretic result. Explicit bounds on the number of points that one can have in $\R^n$ without forming an angle larger than a given bound can be found in \cite{MR841305}. A proof of Lemma~\ref{lma:angles} can also be found in \cite{MR2737262}.
 
 \begin{lemma}\label{lma:angles}
 For any $n \in \N$ and $0 < \beta < \pi$ there is an $N \in \N$ such that if $S \subseteq \mathbb{R}^n$ has cardinality at least $N$ then there are distinct $x,y,z \in S$ such that $\beta \leq \angle(xyz) \leq \pi$.
\end{lemma}
 


\subsection{Proof of Theorem~\ref{thm:main}}\label{sec:thm:main}

The proof of Theorem~\ref{thm:main} as well as the proof of Theorem~\ref{thm:main1} combines two observations: snowflaking forbids the formation of large angles whereas the fact that we have many points forces such angles to exist.
In the proof of Theorem~\ref{thm:main} the special form of the snowflaking function allows us to directly prove a bound on the cardinality of the snowflaked space that can be embedded in the normed vector space. 

%

\begin{proof}[Proof of Theorem~\ref{thm:main}]
Fix $n$ and $\alpha$ and let $C,\epsilon, K$ be given by Lemma~\ref{lma:closetoEuclidean} and Lemma~\ref{lemma234124}.
 Now take $\theta < \min\{\eps,\frac{\pi}{4},\frac{2-2^\alpha}{3C(2K+2^\alpha)}\}$.
 
We let $N \in \N$ be the constant in Lemma~\ref{lma:angles} with $\beta = \pi - \theta$. 
 We suppose towards a contradiction that $(X,d)$ has cardinality $N$ and that there is an isometric embedding $\iota$ of $(X,d^\alpha)$ into an $n$-dimensional normed vector space $(V,\|\cdot\|)$. 
 By Lemma~\ref{lma:angles} there exist three isometrically embedded points $x,y,z \in \iota(X)$ such that
 $\pi - \theta < \angle_z(x,y) \leq \pi$. 
 We declare $\delta_z := \angle_z(x,y).$
 We may assume that $\delta_z \ne \pi$. 
We also declare $\delta_x := \angle_x(y,z)$ and $\delta_y := \angle_y(x,z)$.
 Note that $0 < \delta_x,\delta_y < \theta$.

 Let $z' \in \R^n$ be the orthogonal projection of $z$ on the line passing through $x$ and $y$, i.e. $z' := x + \langle z-x,y-x \rangle(y-x)$.
Lemma~\ref{lma:closetoEuclidean} yields
 \begin{equation}\label{eq:dy}
  \left|\|x-z\| - \|x-z'\|\right| \le C\|x-z'\|\delta_x  
 \end{equation}
 and
 \begin{equation}\label{eq:dx}
  \left|\|y-z\| - \|y-z'\|\right| \le C\|y-z'\|\delta_y.
  \end{equation}
 
 We now estimate
 \[
  (\|y-z\|^\frac{1}{\alpha} + \|z-x\|^\frac{1}{\alpha})^{2\alpha}
 \]
 from above. 
 Subadditivity of $t \mapsto t^\alpha$ yields:
 \begin{align*}
   (\|y-z\|^\frac{1}{\alpha} + \|z-x\|^\frac{1}{\alpha})^{2\alpha} &= ((\|y-z\|^\frac{1}{\alpha} + \|z-x\|^\frac{1}{\alpha})^2)^{\alpha}\\
   &= (\|y-z\|^\frac2\alpha + \|z-x\|^\frac2\alpha + 2\|y-z\|^\frac1\alpha\|z-x\|^\frac1\alpha)^\alpha\\
   & \le \|y-z\|^2 + \|z-x\|^2 + 2^\alpha\|y-z\|\|z-x\|.
 \end{align*}
 Estimating the obtained terms from above using \eqref{eq:dx} and \eqref{eq:dy} we have
 \begin{align*}
  (\|y-z\|^\frac{1}{\alpha} + \|z-x\|^\frac{1}{\alpha})^{2\alpha} & \le 
 \|y-z'\|^2(1+C\delta_y)^2 + \|z'-x\|^2(1+C\delta_x)^2 \\
 & \quad\, + 2^\alpha \|y-z'\|\|z'-x\|(1+C\delta_x)(1+C\delta_y).
 \end{align*}
 Now we use the fact that $C\delta_x < 1$ and Lemma~\ref{lemma234124} to obtain
 \begin{align*}
  \|y-z'\|^2(1+C\delta_y)^2 & \le \|y-z'\|^2(1+3C\delta_y)
  \le \|y-z'\|^2\left(1+3CK\delta_x\frac{\|z'-x\|}{\|z'-y\|}\right)\\
  & \le \|y-z'\|^2 + 3CK\theta\|y-z'\|\|x-z'\|
 \end{align*}
 and similarly
 \[
  \|x-z'\|^2(1+C\delta_x)^2 \le \|x-z'\|^2 + 3CK\theta\|y-z'\|\|x-z'\|.
 \]
 Again by $C\delta_x < 1$ we get $(1+C\delta_x)(1+C\delta_y) < 1+3C\theta$.
 Collecting the estimates together and using the fact that $6CK\theta + 2^\alpha(1+3C\theta) < 2$ we have:
 \begin{align*}
  (\|y-z\|^\frac{1}{\alpha} + \|z-x\|^\frac{1}{\alpha})^{2\alpha} & \le \|y-z'\|^2 + 3CK\theta\|y-z'\|\|x-z'\|\\
  & \quad\, + \|x-z'\|^2 + 3CK\theta\|y-z'\|\|x-z'\|\\
  & \quad\, + 2^\alpha(1+3C\theta) \|y-z'\|\|z'-x\| \\
  & = \|y-z'\|^2 + \|x-z'\|^2\\
  & \quad\, + (6CK\theta + 2^\alpha(1+3C\theta))\|y-z'\|\|z'-x\|\\
  & < \|x-z'\|^2 + \|z'-y\|^2 + 2\|x-z'\|\|z'-y\|\\
   & = (\|x-z'\|+\|z'-y\|)^2 = \|x-y\|^2.
 \end{align*}
%
%
 Therefore
 \begin{align*}
  d(\iota^{-1}(x),\iota^{-1}(y)) & =  \|x-y\|^\frac{1}{\alpha} > \|y-z\|^\frac{1}{\alpha} + \|z-x\|^\frac{1}{\alpha}\\
  & =  d(\iota^{-1}(z),\iota^{-1}(y)) +  d(\iota^{-1}(x),\iota^{-1}(z)).
 \end{align*}
This contradicts the triangle inequality in $(X,d)$.
\end{proof}

\subsection{Proof of Theorem~\ref{thm:main1}}\label{sec:thm:main1}

In the proof of Theorem~\ref{thm:main1} we use the same geometric lemmas (Lemma~\ref{lma:closetoEuclidean} and Lemma~\ref{lemma234124}) as in the proof of Theorem~\ref{thm:main}.
However, in the proof of Theorem~\ref{thm:main1} the choice of a sequence of points giving the contradiction depends not only on the snowflaking function $h$, but also on the first element of the sequence. Therefore no upper bound (depending on $h$ and $n$) on the number of points that can be snowflake embedded can in general be obtained in Theorem \ref{thm:main1}.

\begin{proof}[Proof of Theorem~\ref{thm:main1}]

Suppose to the contrary that $X$ is infinite and that there exists an isometric embedding $\iota \colon (X,h\circ d) \to V$ where $V$ is an $n$-dimensional normed vector space. We divide our proof into two cases.
An infinite bounded subset of $\mathbb{R}^n$ is not discrete, so one of the following holds:
\begin{enumerate}
 \item[(i)] $\iota(X)$ is unbounded; 
   \item[(ii)] $\iota(X)$ is not discrete. 
\end{enumerate}

If (i) holds we will arrive at a contradiction with the condition (S4) of a snowflaking function. If (ii) holds, a contradiction follows with (S3).

\textbf{Case (i): Suppose $\iota(X)$ is unbounded}\\
Observe that (S4) implies that the existence of a function $T \colon \R_{\ge} \to \R_{\ge}$ such 
that for any $t > 0$ and $S \ge T(t)$ we have
\begin{equation}\label{tapio1}
 \frac{t}{h(t)}\frac{h(S)}{S} \le \frac12.
\end{equation}
Combining \eqref{tapio1} with (S1) and (S2) we get 
\begin{equation}\label{tapio2}
 h(S+t) \le h(S) + t\frac{h(S)}{S} = h(S) + \frac{t}{h(t)}\frac{h(S)}{S} h(t)
 \le h(S) + \frac12 h(t).
\end{equation}
Now fix $x_0,x_1 \in X$, $x_0\ne x_1$. Since $(X,d)$ is unbounded, there exists a point $x_2 \in X$ with
 $\angle_{\iota(x_2)}(\iota(x_0), \iota(x_1) )\leq \pi/4$
 and
$d(x_2,x_i) > T(d(x_0,x_1))$ for $i = 0,1$. We continue inductively. Suppose $(x_i)_{i=0}^{N-1} \subset X$ have been chosen.
Now we select $x_N \in X$ satisfying 
\begin{equation}\label{iota_x}
\angle_{\iota(x_N)}(\iota(x_i), \iota(x_j) )\leq \pi/4
\text{  and }
 d(x_i,x_N) > T(d(x_i,x_j)),\quad\forall i,j <N .
 \end{equation}

Let $\epsilon, C , K$ be the constants from Lemma~\ref{lma:closetoEuclidean} and Lemma~\ref{lemma234124}. 
Set 
\begin{equation}\label{eq:delta}
 \delta=\min\{ \eps, \pi/4, \frac{1}{2C(1+K)} \}.
\end{equation}
By Lemma~\ref{lma:angles} there exist 
$x,y,z$ in $\{\iota(x_\ell)\}_{\ell\in \N}$ such that $\angle_{z}(x,y) > \pi - \delta$.
By the condition \eqref{iota_x}, there exist 
$ i, j, k \in \N$ with $k > \max\{i,j\}$ such that
$x=\iota(x_j)$,
$y=\iota(x_k)$,
$z= \iota(x_i)$.
Let $z'$ be the orthogonal projection of $z$ to the line passing through
$x$ and $y$. On the one hand, we have that
\begin{align*}
 h(d(x_j,x_k)) & = \|x-y\| \\
 & = \|x-z'\| + \|z'-y\|\\
(\text{Using  Lemma~\ref{lma:closetoEuclidean}})\quad & \ge \|x - z\| - C\|x-z'\|\angle_{x}(z,y)\\
 & \hspace{2cm} +  \|z - y\| - C\|y-z'\|\angle_{y}(z,x)\\
 (\text{Using  Lemma~\ref{lemma234124}})\quad & \ge \|x - z\| - C\|x-z'\|\angle_{x}(z,y)\\
 & \hspace{2cm} +  \|z - y\| - CK\|x-z'\|\angle_{x}(z,y)\\
 & = \|x - z\| + \|z - y\| - C(1+K)\|x - z'\|\angle_{x}(z,y)\\
  & \ge \|x - z\| + \|z - y\| - C(1+K)\|x - z\|\angle_{x}(z,y)\\
 & = h(d(x_i,x_k)) + (1-C(1+K) \angle_{x}(z,y)) h(d(x_i,x_j))\\
 (\text{Using  that } \angle_{x}(z,y)< \delta)   \quad & > h(d(x_i,x_k)) + \frac12 h(d(x_i,x_j)).
\end{align*}
On the other hand, 
first notice that since 
$h$ is concave and positive, then $h$ has to be weakly increasing.
Secondly, by the definition of the function $T$ we then have that for any three $ i, j, k \in \N$ with $k > \max\{i,j\}$  
\begin{equation}\label{eq:a_bit_shorter}
 h(d(x_j,x_k)) \le h(d(x_i,x_k)+d(x_i,x_j)) \le h(d(x_i,x_k)) + \frac12 h(d(x_i,x_j)).
\end{equation}
Therefore we have a contradiction.


\textbf{Case (ii): Suppose $\iota(X)$ is not discrete}\\
This time we observe that (S3) implies the existence of a function $\tilde T \colon \R_{\ge} \to \R_{\ge}$ such 
that $\tilde T(r) \le r$ for all $r$ and for any $S > 0$ and $0 < t \le \tilde T(S)$ we have
\eqref{tapio1}, and hence \eqref{tapio2}, using (S1) and (S2). 


Let $y$ be an accumulation point of $X$. First we select $x_0 \in X \setminus \{y\}$.
Next we take a radius $r_0 > 0$ so that for all $y_1,y_2 \in B(y,r_0)$ we have both
$\angle_{\iota(x_0)}(\iota(y_1), \iota(y_2) )\leq \pi/4$
 and
$d(y_1,y_2) < \tilde T(d(x_0,y_i))$ for $i = 0,1$. Now we select a point $x_1 \in B(y,r_1) \setminus \{y\}$.
We continue inductively. Suppose $(x_i)_{i=0}^{N-1} \subset X$ have been chosen.
Now we take a radius $r_{N-1} < r_{N-2}$ such that for all $y_1,y_2 \in B(y,r_0)$ we have
\begin{equation}\label{iota_x_new}
\angle_{\iota(x_i)}(\iota(y_1), \iota(y_2) )\leq \pi/4
\text{  and }
 d(y_1,y_2) < \tilde T(d(x_i,y_j)),\quad\forall i <N, j = 1,2.
 \end{equation}
Then we select a point $x_N \in B(y,r_{N-1}) \setminus \{y\}$.

With the points $\{x_i\}$ we now arrive at a contradiction with the same argument as in the case (i).
Let $\delta$ be as in \eqref{eq:delta}. Again by Lemma~\ref{lma:angles} there exist 
$x,y,z$ in $\{\iota(x_\ell)\}_{\ell\in \N}$ such that $\angle_{z}(x,y) > \pi - \delta$,
but this time by the condition \eqref{iota_x_new}, there exist 
$ i, j, k \in \N$ with $k < \min\{i,j\}$ such that
$x=\iota(x_j)$,
$y=\iota(x_k)$,
$z= \iota(x_i)$. Now we continue verbatim the proof in the case (i).
\end{proof}

\subsection{Necessity of (S3) and (S4)}

We end this paper by showing that the conditions (S3) and (S4) of generalized snowflakes are indeed needed for Theorem \ref{thm:main1} to hold.

\begin{proposition}\label{prop:necessity}
 Suppose $h \colon \R_\ge \to \R_\ge$ satisfies (S1) and (S2) but fails to satisfy (S3) or (S4). Then there is an
 infinite metric space $(X,d)$ such that $(X,h \circ d)$ admits an isometric embeddeding into $2$-dimensional Euclidean space.
\end{proposition}
\begin{proof}
 We only treat the case where (S4) fails. The case where (S3) fails follows in a similar way.
 
We construct a sequence of points $\{x_i\}_{i = 0}^{\infty}$ in $\R^2$ such that $(X,d) := (\{x_i\}_{i = 1}^{\infty},h^{-1}\circ d_E)$
 is a metric space. For this purpose we fix a sequence $\{\alpha_i\}_{i=1}^\infty$ of positive angles such that $\sum_{i=1}^\infty \alpha_i < \frac{\pi}{2}$. 
 Depending on the sequence $\{\alpha_i\}_{i = 1}^{\infty}$ and the function $h$ we construct an increasing sequence $\{t_i\}_{i=1}^\infty$ of positive real numbers which determine the Euclidean distance between $x_{i-1}$ and $x_i$. 
 For notational convenience we let $c(t) = \frac{h(t)}t$. Notice that by assumption $c(t) \searrow c > 0$ as $t \to \infty$.
 This allows us to select for every $i \in \N$ a real number $t_i > 0$ such that for all $s,t \ge t_i$ we have
 \begin{equation}\label{eq:tichoice}
  c(s)(c(s+t)-c) + c(t)(c(s+t)-c) \le 2(c(s)c(t) \cos(\pi-\alpha_i) + c(s+t)^2).
  \end{equation}

 Now, using the sequences $\{\alpha_i\}_{i = 1}^{\infty}$ and $\{t_i\}_{i = 1}^{\infty}$ we define the sequence $\{x_i\}_{i \in \mathbb{N}}$ as follows.
 We set 
 $x_0 := (0,0)$, $x_1 := (t_1,0)$, and  inductively for $n \geq 2$ declare
 \[
  x_n := x_{n-1} + \left(t_n\sin(\sum_{j=1}^{n-1}a_j),t_n\cos(\sum_{j=1}^{n-1}a_j)\right).
 \]

 In order to see that $(X,d)$ is a metric space we need to check that the triangle inequality holds. For this purpose let $0 \le i < j < k$ be three integers. Let $d_E$ be the Euclidean metric.
 The only nontrivial inequality that we have to verify is
 \[
  h^{-1}(d_E(x_i,x_k)) \le   h^{-1}(d_E(x_j,x_k)) +   h^{-1}(d_E(x_j,x_k)).
 \]
 Denoting $s :=   h^{-1}(d_E(x_i,x_j))$ and $t :=   h^{-1}(d_E(x_j,x_k))$ the above inequality is equivalent to
 \begin{equation}\label{eq:wewant}
  d_E(x_i,x_k) \le h(s+t).  
 \end{equation}
 By (S1) and (S2) we can estimate
 \begin{equation}\label{eq:cestim}
  c(t) = \frac{h(t)}t \le \frac{h(t+s) - cs}{t} = c(t+s) + (c(t+s) - c)\frac{s}{t}.
 \end{equation}
 Since $s,t \ge t_j$, by applying the law of cosines, 
 \eqref{eq:cestim} and \eqref{eq:tichoice} we obtain
 \begin{align*}
  d_E(x_i,x_k)^2 - h(s+t)^2 & = d_E(x_i,x_j)^2 + d_E(x_j,x_k)^2 \\
  & \qquad - 2d_E(x_i,x_j)d_E(x_j,x_k)\cos(\angle_{x_j}(x_i,x_k)) - h(s+t)^2\\
  & \le h(s)^2 + h(t)^2 - 2h(s)h(t)\cos(\pi-\alpha_j) - h(s+t)^2\\
  & = s^2c(s)^2 + t^2c(t)^2 - 2stc(s)c(t) \cos(\pi-\alpha_j) - (s+t)^2c(s+t)^2\\
  & = s^2(c(s)^2-c(s+t)^2) + t^2(c(t)^2 - c(s+t)^2)\\
 & \qquad - 2st(c(s)c(t) \cos(\pi-\alpha_j) + c(s+t)^2)\\
 \text{(Using \eqref{eq:cestim})} \qquad  & \le st\big(c(s)(c(s+t)-c) + c(t)(c(s+t)-c)\\
  & \qquad- 2(c(s)c(t) \cos(\pi-\alpha_j) + c(s+t)^2)\big)\\
  \text{(Using \eqref{eq:tichoice})} \qquad & \le 0
 \end{align*}
 and thus \eqref{eq:wewant} holds.
\end{proof}

\begin{remark}\label{rmk:nobound}
 The proof of Proposition \ref{prop:necessity} can be modified to show that there is a snowflake function $h$ (satisfying all the conditions (S1)--(S4)) such that for every $n \in \N$ there exists a metric space $(X_n,d_n)$ with cardinality $n$ so that $(X_n,h\circ d_n)$ embeds 
 isometrically into $2$-dimensional Euclidean space.
 
 Indeed, suppose we are given $n \in \N$ and we have already defined $h$ on $[0,T_{n-1}]$ and that the slope of $h$ at $T_{n - 1}$ is $c_{n-1} > 0$. Then we can define $h$ on an arbitrary long interval $[T_{n-1},T_n]$ as $h(t) = H(T_{n-1}) + c_nt$, where $0 < c_n < c_{n-1}$. By the proof of Proposition \ref{prop:necessity},
 by taking $S$ large enough, there exist $n$ points in $\R^2$ with the distance between any two of them between $h(T)$ and $h(S)$ so that 
 they are an $h$-snowflake of some metric space. We can define $h$ near $0$ so that it satisfies (S3) and if we take $c_n \searrow 0$ then (S4) is also satisfied. Therefore $h$ will have all the required properties.
\end{remark}

 \bibliography{LRW_biblio}
\bibliographystyle{amsalpha}

\end{document}